\documentclass[11pt,a4paper]{article}
\usepackage[utf8]{inputenc}
\usepackage[english]{babel}
\usepackage{amsmath}
\usepackage{amsthm}
\usepackage{amsfonts}
\usepackage{amssymb}
\usepackage{array}  
\usepackage{graphicx} 
\usepackage{color}
\usepackage{mathrsfs}
\usepackage{hyperref}
\hypersetup{linktoc=page,linkcolor=red,citecolor=blue,
filecolor=blue,urlcolor=cyan,colorlinks=true}   
\usepackage{subcaption}
\usepackage{multicol}
\usepackage{lipsum}
\usepackage{geometry}
\usepackage{mdwlist}	

\theoremstyle{plain}
\newtheorem{theorem}{Theorem}

\newtheorem{lemma}[theorem]{Lemma}

\theoremstyle{definition}
\newtheorem{definition}[theorem]{Definition}

\newtheorem{remark}[theorem]{Remark}

\newcommand{\E}{\mathbb{E}}
\newcommand{\G}{\mathbb{G}}
\newcommand{\V}{\mathbb{V}}

\newcommand{\Z}{\mathbb{Z}}
\newcommand{\Pro}{\mathbb{P}}

\newcommand{\ee}{\textbf{\textup{e}}}

\newcommand{\calC}{\mathcal{C}}

\newcommand{\calG}{\mathcal{G}}

\title{Exponential decay of the volume for Bernoulli percolation: a proof via stochastic comparison}
\date{}
\author{Hugo Vanneuville\thanks{CNRS and Institut Fourier (Université Grenoble Alpes), hugo.vanneuville@univ-grenoble-alpes.fr.}
}

\begin{document}

\maketitle

\abstract{
Let us consider subcritical Bernoulli percolation on a connected, transitive, infinite and locally finite graph. In this paper, we propose a new (and short) proof of the exponential decay property for the volume of clusters. We do not rely on differential inequalities and rather use stochastic comparison techniques, which are inspired by several works including the paper \textit{An approximate zero-one law} written by Russo in the early eighties.
}
{\footnotesize

\tableofcontents
}

\section{Introduction}

Let us consider \textit{Bernoulli bond percolation} on a connected, locally finite, vertex-transitive and countably infinite graph $\G=(\V,\E)$ (e.g.\ the hypercubic lattice $\Z^d$). `Locally finite' means that the degree of each vertex is finite, and `vertex-transitive' means that for all vertices $v,w$, there exists an automorphism of $\G$ that maps~$v$ to~$w$.

\smallskip

This model is defined as follows: each edge is declared \textit{open} with some probability $p$ and \textit{closed} otherwise, independently of the other edges. We let $\Pro_p$ denote the corresponding product probability measure on $\{0,1\}^{\E}$, where $1$ means open and $0$ means closed. In percolation theory, one is interested in the connectivity properties of the graph obtained by keeping only the open edges. We fix a vertex $o \in \V$ once and for all, we let $\calC_o$ denote the cluster of $o$, i.e.\ the set of all vertices that are connected to $o$ by an open path, and we let $|\calC_o|$ denote its cardinality. The \textit{critical probability} is defined as follows:
\[
p_c = \inf \big\{ p \in [0,1] : \Pro_p\big[|\calC_o|=+\infty\big]>0 \big\}.
\]
Let
\[
\psi_n(p) = \Pro_p\big[|\calC_o| \ge n\big], \quad \forall n \ge 0.
\]

In this paper, we give a new proof of the following theorem, which states that the volume of $\calC_o$ has an exponential tail in the subcritical regime.

\begin{theorem}\label{thm:exp'}
For every $p<p_c$, there exist $c,C>0$ such that
\[
\forall n \ge 0, \quad \psi_n(p) \le Ce^{-cn}.
\]
\end{theorem}

Theorem \ref{thm:exp'} was proven independently by Menshikov \cite{Men86} and Aizenman and Barsky \cite{AB87}.\footnote{\cite{Men86,AB87} consider graphs with subexponential growth; see \cite{AV08} for the extension to graphs with exponential growth.} More precisely, they proved that $\Pro_p[\textup{diam}(\calC_o)\ge n]$ decays exponentially fast in the subcritical regime (where $\textup{diam}$ denotes the diameter for the graph distance), and Kesten \cite{Kes81} and Aizenman and Newman \cite{AN84} had previously shown that this was equivalent to Theorem \ref{thm:exp'}. We refer to \cite{DT16,DT17,DRT19,Van22} for more recent proofs of exponential decay of $\Pro_p[\textup{diam}(\calC_o)\ge n]$ and to \cite{Hut20} for a direct proof of Theorem \ref{thm:exp'}. Concerning the more recent proofs, Duminil-Copin and Tassion \cite{DT16,DT17} have proposed a very short proof with a `branching processes flavour' and Duminil-Copin, Raoufi and Tassion \cite{DRT19} have proposed a proof via the so-called OSSS inequality, that extends to several dependent percolation models. For more about these results -- and percolation in general --, see for instance the books \cite{Gri99,BR06} or the lecture notes \cite{Dum17}.

\smallskip

The study of differential inequalities is central in all the proofs of exponential decay, except in the (not for publication) work \cite{Van22}. In \cite{Van22}, we studied $\Pro_p[\textup{diam}(\calC_o)\ge n]$ by using stochastic comparison techniques inspired both by a work of Russo in the early 80's (see~\cite[Section~3]{Rus82}) and by works that combine exploration procedures and coupling techniques in percolation theory (see e.g.\ \cite{DM22}, see also~\cite{GS14} for more about the application of exploration procedures to percolation theory). In the present paper -- which can be read independently of \cite{Van22} --, drawing inspiration from a recent work by Easo \cite{Eas22},\footnote{Among other things, Easo adapts some techniques from \cite{Van22} to the study of the volume of percolation clusters on finite transitive graphs.} we observe that these stochastic comparison techniques actually seem more suitable to study the volume than the diameter, and we propose a short and direct proof of Theorem~\ref{thm:exp'}.

\medskip

To study the volume of clusters, it is relevant to introduce what is often referred to as a `ghost field'. By this, we just mean that, given some parameter $h \in (0,+\infty)$, we color every vertex \textit{green} with probability $1-e^{-h}$, independently of the other vertices and of the bond percolation configuration. We let $\Pro_{p,h}$ denote the product probability measure on $\{0,1\}^{\E} \times \{0,1\}^{\V}$ with this additional feature, where we assign the number $1$ to every green vertex, i.e.
\[
\Pro_{p,h} = \big( p\delta_1+(1-p)\delta_0\big)^{\otimes \E} \otimes \big( (1-e^{-h})\delta_1+e^{-h}\delta_0\big)^{\otimes \V}.
\]
We let $\calG$ denote the set of green vertices and we consider the so-called \textit{magnetization} (this terminology comes from a comparison with analogous quantities that appear in the Ising model):
\[
m_h(p)=\Pro_{p,h} \big[ \calC_o \cap \calG \ne \emptyset \big].
\]
One can note that, for every $p$, $m_h(p)$ goes to $\Pro_p[|\calC_o|=+\infty]$ as $h$ goes to $0$.

\medskip

Theorem \ref{thm:exp'} is a direct consequence of the following \textit{near-critical} sharpness result, which can also be deduced from \cite{Hut20} (with different constants).

\begin{theorem}\label{thm:exp}
Let $p\in(0,1)$ and $h\in(0,+\infty)$, and let $q=p(1-m_h(p))$. Then,
\[
\forall n \ge 0, \quad \psi_n(q) \le \frac{1}{1-m_h(p)} \psi_n(p) e^{-hn}.
\]
\end{theorem}

\begin{proof}[Proof of Theorem \ref{thm:exp'} by using Theorem \ref{thm:exp}]
Let $p<p_c$. Let $p'$ be any number in $(p,p_c)$. Since $m_h(p')$ goes to $0$ as $h$ goes to $0$, we can -- and we do -- choose some $h\in(0,+\infty)$ such that $p'(1-m_h(p'))\ge p$. Then,
\[
\psi_n(p) \overset{\text{Thm \ref{thm:exp}}}{\le} \frac{1}{1-m_h(p')}\psi_n(p')e^{-hn} \le Ce^{-hn},
\]
with $C=1/(1-m_h(p'))$.
\end{proof}

Let us note that Theorem \ref{thm:exp} is the analogue of \cite[Theorem 1.1]{Van22} for the volume. Surprisingly (maybe), the proof of the former is simpler. Theorem \ref{thm:exp} also implies the so-called \textit{mean-field lower bound}, that was proven by Chayes and Chayes~\cite{CC87}:

\begin{theorem}\label{thm:mf}
For every $p\ge p_c$,
\[
\Pro_p\big[|\calC_o|=+\infty\big]\ge (p-p_c)/p.
\]
\end{theorem}

\begin{proof}[Proof of Theorem \ref{thm:mf} by using Theorem \ref{thm:exp}]
Fix some $p\ge p_c$. For every $h>0$, let $q_h=p(1-m_h(p))$. By Theorem \ref{thm:exp}, $\psi_n(q_h)\rightarrow 0$ as $n\rightarrow+\infty$. As a result, $q_h\le p_c$. The result follows by letting $h$ go to $0$ and using that $m_h(p)$ converges to $\Pro_p[|\calC_o|=~+\infty]$.
\end{proof}

\noindent \textbf{Organization of the rest of the paper:}
\begin{itemize*}
\item Some ideas behind the proof are provided in Section \ref{sec:ideas};
\item The proof is written in Section \ref{sec:gen} (where we propose a general criterion that is not specific to percolation theory) and in Section~\ref{sec:proof} (that contains the proof of Theorem \ref{thm:exp});
\item In Section \ref{sec:motiv}, we discuss our motivations as well as the possibility to prove new results via the stochastic domination techniques proposed in the present paper.
\end{itemize*}

\medskip

\noindent \textbf{Acknowledgments:} I am extremely grateful to Philip Easo for very inspiring discussions and for his work~\cite{Eas22}. I also thank Sébastien Martineau and Christophe Leuridan for comments that helped a lot to improve this manuscript, Tom Hutchcroft for interesting discussions about \eqref{eq:in_exp}, Vincent Beffara and Stephen Muirhead for discussions that helped to simplify some parts of the proof, Barbara Dembin and Jean-Baptiste Gouéré for inspiring feedbacks, and Hugo Duminil-Copin, who encouraged me to include a discussion about possible new results. Moreover, I thank Loren Coquille who noticed that there was still a (discrete) differential inequality flavour in \cite{Van22} -- such a thing does not appear anymore in the present paper. I would like to thank Vincent Tassion, who recommended to combine coupling techniques from~\cite{Van22,Eas22} with ghost field techniques, which led to the present paper. Finally, I would like to thank the anonymous referees for helpful comments.

\section{Some ideas behind the proof}\label{sec:ideas}

How can one control the volume of subcritical clusters? One can first notice that Theorem \ref{thm:exp'} is equivalent to the fact that decreasing $p$ has the following \textit{regularizing effect}:
\begin{center}
If, for a given $p_0$, $|\calC_o|$ is finite $\Pro_{p_0}$-a.s., then for every $p<p_0$, $|\calC_0|$ has an exponential tail under $\Pro_p$.
\end{center}
To prove this, we will show that, given any $\varepsilon>0$ and $p_0$ such that $\Pro_{p_0}\big[|\calC_o|=~+\infty]=0$, \textit{decreasing $p_0$ by $\varepsilon$ has more effect than conditioning on some well-chosen disconnection event~$A$}, essentially in the sense that $\Pro_{p_0-\varepsilon}$ is stochastically smaller than $\Pro_{p_0}[\,\cdot\mid A]$. This disconnection event will be $A=\{ \calC_o \cap \calG = \emptyset\}$ for some parameter $h$.

\smallskip

So the question is now: How can one compare the effect of conditioning on an event to that of slightly decreasing $p$? One can actually do this by exploring a configuration with law $\Pro_{p_0}[\, \cdot \mid A]$ and estimating (at every step of the exploration) the probability that the next edge to be explored has some influence on $A$. If this probability is always small, then slightly decreasing $p$ will have more effect than conditioning on~$A$.

\section{A general criterion to compare the effect of decreasing $p$ to that of conditioning on an event}\label{sec:gen}

\paragraph{A. Explorations.} Let us fix a finite graph $G=(V,E)$ and call $\Pro_p^G$ and $\Pro_{p,h}^{G}$ the analogues of $\Pro_p$ and $\Pro_{p,h}$ on the graph $G$.\footnote{I.e.\ $\Pro_p^G$ is the product Bernoulli measure of parameter $p$ on $\{0,1\}^E$ and $\Pro_{p,h}^G$ is the product Bernoulli measure on $\{0,1\}^E\times\{0,1\}^V$, of parameter $p$ on the edges and $1-e^{-h}$ on the vertices. The proofs in this section actually work in a more general setting, for instance by considering a probability space $(M,\mathcal{F},\rho)$ and considering the measure $\Pro_p^G\otimes\rho$ instead of $\Pro_{p,h}^G$.} We denote by $\omega$ the elements of $\{0,1\}^E$ and by $\eta$ the elements of $\{0,1\}^V$. Moreover, we let $\vec{E}$ be the set of all orderings $(e_1,\dots,e_{|E|})$ of $E$.

\begin{definition}\label{defi:expl}
An \textit{exploration} of $E$ is a map
\[
\begin{array}{rl}
\ee : \{0,1\}^E & \longrightarrow \vec{E}\\
\omega & \longmapsto (\ee_1,\dots,\ee_{|E|})=\big(\ee_1(\omega),\dots,\ee_{|E|}(\omega)\big)
\end{array}
\]
such that $\ee_1$ does not depend on $\omega$ and $\ee_{k+1}$ only depends on $(\ee_1,\dots,\ee_k)$ and $(\omega_{\ee_1},\dots,\omega_{\ee_k})$.\footnote{This means that for every $k \in \{0,\dots,|E|-1\}$, there exists a function $\phi_k$ such that $\forall \omega \in \{0,1\}^E, \; \ee_{k+1}(\omega)=\phi_k\big(\ee_1(\omega),\dots,\ee_k(\omega),\omega_{\ee_1(\omega)},\dots,\omega_{\ee_k(\omega)}\big)$.} Given an exploration $\ee$, $(e,x) \in \vec{E}\times\{0,1\}^E$ and $k \in \{0,\dots,|E|\}$, we denote by $\textup{Expl}_k(e,x)$ the event that $(\ee,\omega)$ coincides with $(e,x)$ at least until step $k$, i.e.\
\[
\textup{Expl}_k(e,x)=\big\{\omega \in \{0,1\}^E : \forall j \in \{1,\dots,k\}, \ee_j=e_j\text{ and }\omega_{e_j}=x_{e_j}\big\}.
\]
(In particular, $\textup{Expl}_0(e,x)=\{0,1\}^E$.) We also let $\textup{Expl}(e,x)=\textup{Expl}_{|E|}(e,x)$.
\end{definition}

We recall that an \textit{increasing} subset $A\subseteq\{0,1\}^E$ is a subset that satisfies ($\omega \in A$ and $\forall e \in E, \quad \omega'_e \ge \omega_e$) $\Rightarrow$ $\omega'\in A$. Moreover, given two probability measures $\mu,\nu$ on $\{0,1\}^E$, we say that $\mu$ is \textit{stochastically smaller} than $\nu$ if $\mu[A]\le \nu[A]$ for every increasing subset $A$, and we denote this property by $\mu \preceq \nu$. Finally, we say that an edge $e \in E$ is \textit{pivotal} for a set $A \subseteq \{0,1\}^E \times \{0,1\}^V$ and an element~$(\omega,\eta)$ if changing only the edge $e$ in $\omega$ (and not changing $\eta$) modifies~$1_A(\omega,\eta)$.

\paragraph{B. The main intermediate lemma.} The key result of this section is the following lemma. We fix an exploration of $E$ for the rest of the section.

\begin{lemma}\label{lem:coupl}
Let $p \in (0,1)$ and $h \in (0,+\infty)$. Moreover, let $A \subseteq \{0,1\}^E \times \{0,1\}^V$ be any non-empty set and let $\varepsilon \in [0,1]$. Assume that for every $(e,x) \in \vec{E}\times\{0,1\}^E$ such that $A \cap \textup{Expl}(e,x)\neq \emptyset$, we have
\[
\forall k \in \{0,\dots,|E|-1\}, \quad \Pro_{p,h}^G \big[ \text{$e_{k+1}$ is pivotal for $A$} \; \big| \; A \cap \textup{Expl}_k(e,x) \big] \le \varepsilon.
\]
Then,
\[
\Pro_{p(1-\varepsilon)}^G \preceq \Pro_{p,h}^G \big[ \omega \in \cdot \; \big| \; A \big].
\]
\end{lemma}

\begin{remark}\label{rk:formal}
\begin{itemize}
\item Recall that we use the letter $\omega$ to denote the elements of $\{0,1\}^E$. So $\Pro_{p,h}^G [ \omega \in \cdot  \mid A]$ denotes the marginal law on $\{0,1\}^E$ of $\Pro_{p,h}^G[ \, \cdot \mid A]$;
\item We have written $\textup{Expl}_k(e,x)$ when we should have written $\textup{Expl}_k(e,x) \times \{0,1\}^V$;
\item Let us take the opportunity of this remark to make the following observation that holds when $(e,x)$ is in the image of $\omega \mapsto (\ee(\omega),\omega)$: In this case, $\textup{Expl}_k(e,x)=\{\omega \in \{0,1\}^E : \forall j \in \{1,\dots,k\}, \omega_{e_j}=x_{e_j}\}$. As a result, $\Pro_{p,h}^G [ \, \cdot \mid \textup{Expl}_k(e,x) ]$ is the probability measure on $\{0,1\}^E\times\{0,1\}^V$ that assigns the value $x_{e_j}$ to $e_j$ for every $j \le k$ and is still the product Bernoulli measure of parameter $p$ on the other edges and of parameter $1-e^{-h}$ on the vertices.
\end{itemize}
\end{remark}

The proof of Lemma \ref{lem:coupl} is based on the following lemma (see e.g.\ \cite[Lemma~1]{Rus82}, \cite[(7.64)]{Gri99} and \cite[Lemma 2.1]{DRT19} for very similar results).

\begin{lemma}\label{lem:gen}
Let $q \in [0,1]$. Moreover, let $\mu$ be a probability measure on $\{0,1\}^E$ and assume that for every $(e,x) \in \vec{E}\times\{0,1\}^E$ such that $\mu\big[\textup{Expl}(e,x)\big]>0$, we have
\[
\forall k \in \{0,\dots,|E|-1\}, \quad \mu \big[ \omega_{e_{k+1}} = 1 \; \big| \; \textup{Expl}_k(e,x) \big] \ge q.
\]
Then, $\Pro_q^G \preceq \mu$.
\end{lemma}

\begin{proof}[Proof of Lemma \ref{lem:gen}]
The result can be proven by induction on $|E|$. Throughout the proof, it is important to remember that $\ee_1$ is constant. Let $A\subseteq\{0,1\}^E$ be an increasing set. We want to prove that $\mu[A]\ge\Pro^G_q[A]$. We extend every $\omega \in \{0,1\}^{E\setminus\{\ee_1\}}$ into two configurations $\omega^0,\omega^1\in\{0,1\}^E$ by setting $\omega^i_{\ee_1}=i$. We define $A^0,A^1\subseteq\{0,1\}^{E\setminus\{\ee_1\}}$ by $A^i=\{\omega \in \{0,1\}^{E\setminus\{\ee_1\}} : \omega^i \in A\}$. Moreover, we define two probability measures $\mu^0,\mu^1$ on $\{0,1\}^{E\setminus\{\ee_1\}}$ by $\mu^i=\mu[\omega_{|E\setminus\{\ee_1\}}\in\cdot\mid\omega_{\ee_1}=i]$.

\smallskip

Let $G'$ be the graph obtained from $G$ by erasing $\ee_1$. By the induction hypothesis -- applied to the explorations $\omega\in\{0,1\}^{E\setminus\{\ee_1\}}\mapsto(\ee_2(\omega^i),\dots,\ee_{|E|}(\omega^i))$ --, we have $\mu^i[A^i]\ge \Pro^{G'}_q[A^i]$ for every $i\in \{0,1\}$. As a result,
\begin{multline*}
\mu[A]=\mu[\omega_{\ee_1}=0]\mu^0[A^0]+ \mu[\omega_{\ee_1}=1]\mu^1[A^1]\\
\ge \mu[\omega_{\ee_1}=0]\Pro_q^{G'} \big[A^0\big]+ \mu[\omega_{\ee_1}=1]\Pro_q^{G'} \big[A^1\big].
\end{multline*}
By the assumption of the lemma for $k=0$, we have $ \mu[\omega_{\ee_1}=1] \ge q$. Moreover, $A^0\subseteq A^1$ since $A$ is increasing, so the above is larger than or equal to
\[
(1-q)\Pro_q^{G'} \big[A^0\big]+q \Pro_q^{G'} \big[A^1\big] = \Pro^G_q[A]. \qedhere
\]
\end{proof}

\begin{proof}[Proof of Lemma \ref{lem:coupl}]
Let us prove that the hypothesis of Lemma \ref{lem:gen} holds with $\mu=\Pro_{p,h}^G[\omega \in \cdot \mid A]$ and $q=p(1-\varepsilon)$. Let $(e,x)$ as in the statement of Lemma \ref{lem:coupl}. The third point of Remark \ref{rk:formal} implies that, under $\Pro_{p,h}^G \big[ \, \cdot \mid  \textup{Expl}_k(e,x) \big]$, $\omega_{e_{k+1}}$ is independent of the event $A \cap \{ \text{$e_{k+1} $ is not piv. for $A$} \}$ (because the latter depends only on $\eta$ and $\omega_{|E\setminus\{e_{k+1}\}}$). As a result,
\begin{multline*}
\Pro_{p,h}^G \big[ \omega_{e_{k+1}} = 1 \; \big| \; A \cap \textup{Expl}_k(e,x) \big]\\
\ge \frac{\Pro_{p,h}^G \big[ \omega_{e_{k+1}} = 1, A, \text{$e_{k+1} $ is not piv. for $A$} \; \big| \;  \textup{Expl}_k(e,x) \big]}{\Pro_{p,h}^G \big[ A \; \big| \; \textup{Expl}_k(e,x) \big]}\\
= p \times\Pro_{p,h}^G \big[ \text{$e_{k+1} $ is not piv. for $A$} \; \big| \; A \cap \textup{Expl}_k(e,x) \big]\ge p(1-\varepsilon).
\end{multline*}
We end the proof by applying Lemma \ref{lem:gen}.
\end{proof}

\section{Proof of Theorem \ref{thm:exp}}\label{sec:proof}

Let $G_n=(V_n,E_n)$ denote the graph $\G$ restricted to the ball of radius $n$  around~$o$ (for the graph distance). We keep the notations~$\calC_o$ and~$\calG$ (the cluster of $o$ and the ghost field) in the context of percolation on $G_n$. As in Section \ref{sec:gen}, $\Pro_{p}^{G_n}$ and $\Pro_{p,h}^{G_n}$ are the analogues of $\Pro_p$ and $\Pro_{p,h}$ on the graph $G_n$, and we denote by $\omega$ the elements of $\{0,1\}^{E_n}$.

\begin{lemma}\label{lem:coupl_perco}
Let $p \in (0,1)$ and $h \in (0,+\infty)$, and write $q=p(1-m_h(p))$. Then,
\[
\forall n \ge 0, \quad \Pro_{q}^{G_n} \preceq \Pro_{p,h}^{G_n} \big[ \omega \in \cdot \; \big| \; \calC_o \cap \calG = \emptyset \big].
\]
\end{lemma}

\begin{proof}
Let us apply Lemma \ref{lem:coupl}, with $A=\{\calC_o\cap\calG=\emptyset\}$ and $\varepsilon=m_h(p)$. To this purpose, we define an exploration of $E_n$ by  fixing an arbitrary ordering of $E_n$ and:
\begin{itemize*}
\item By first revealing $\calC_o$ (i.e.\ by revealing iteratively the edge of smallest index among all the edges that are connected to $o$ by already revealed open edges). Here and below, `smallest index' refers to the arbitrary ordering that we have fixed;
\item Then, by revealing all the other edges (once again by revealing iteratively the edge of smallest index).
\end{itemize*}
Let $(e,x) \in \vec{E}_n\times\{0,1\}^{E_n}$ such that $\{\calC_o\cap\calG=\emptyset\}\cap\textup{Expl}(e,x)\neq\emptyset$. We observe that, if $\textup{Expl}_k(e,x)$ holds and $e_{k+1}$ is pivotal for $\{\calC_o \cap \calG = \emptyset\}$, then one can write $e_{k+1}=\{v_{k+1},w_{k+1}\}$ where:
\begin{itemize*}
\item $v_{k+1}$ is connected to $o$ by open edges in $\{e_1,\dots,e_k\}$ but $w_{k+1}$ is not;
\item $w_{k+1}$ is connected to a green vertex by open edges that do not belong to $\{e_1,\dots,e_{k+1}\}$.
\end{itemize*}
We let $B_{k+1}$ denote the event of the second item. By this observation and then the Harris--FKG inequality\footnote{If one does not want to use the Harris--FKG inequality in this paper, one can continue to explore the configuration with the same rule as before, except that $e_{k+1}$ is not revealed until no other unrevealed edge is connected to $o$, and one can compute the pivotal probability conditionally on this further information.} (see for instance \cite[Theorem 2.4]{Gri99} or \cite[Lemma~3]{BR06}) applied to $\Pro_{p,h}^{G_n} \big[ \, \cdot \mid \textup{Expl}_k(e,x)\big]$ (which is a product of Bernoulli laws by the third point of Remark \ref{rk:formal}) and to the increasing event $B_{k+1}$ and the decreasing event $\{\calC_o\cap\calG=\emptyset\}$, we obtain that
\begin{multline*}
\Pro_{p,h}^{G_n} \big[\text{$e_{k+1}$ is piv.\ for $\{\calC_o \cap \calG = \emptyset\}$} \; \big| \; \calC_o \cap \calG = \emptyset, \textup{Expl}_k(e,x) \big]\\
\le \Pro_{p,h}^{G_n} \big[ B_{k+1} \; \big| \; \calC_o \cap \calG = \emptyset, \textup{Expl}_k(e,x) \big]\\
\overset{\text{Harris--FKG}}{\le} \Pro_{p,h}^{G_n} \big[ B_{k+1} \; \big| \; \textup{Expl}_k(e,x) \big]= \Pro_{p,h}^{G_n} \big[ B_{k+1} \big] \le m_h(p),
\end{multline*}
where in the equality we have used the third point of Remark \ref{rk:formal}. We conclude by applying Lemma \ref{lem:coupl}.
\end{proof}

\begin{proof}[Proof of Theorem \ref{thm:exp}]
We note that $\psi_n(q)=\Pro_q^{G_n} \big[ |\calC_o|\ge n \big]$. As a result, Lemma \ref{lem:coupl_perco} implies that
\begin{align*}
\psi_n(q)\le\Pro_{p,h}^{G_n} \big[ |\calC_o|\ge n \; \big| \; \calC_o \cap \calG = \emptyset
 \big] &\overset{\text{Bayes}}{=} \frac{\psi_n(p) \times \Pro_{p,h}^{G_n}\big[ \calC_o \cap \calG = \emptyset \; \big| \; |\calC_o| \ge n \big]}{1-\Pro_{p,h}^{G_n}\big[ \calC_o \cap \calG \neq \emptyset \big]}\\
& \hspace{0.2cm}=  \frac{\psi_n(p) \times \E_{p,h}^{G_n}\big[ e^{-h|\calC_o|} \; \big| \; |\calC_o| \ge n \big]}{1-\Pro_{p,h}^{G_n}\big[ \calC_o \cap \calG \neq \emptyset \big]}\\
& \hspace{0.2cm}\le \frac{\psi_n(p)}{1-m_h(p)} e^{-hn}. \qedhere
\end{align*}
\end{proof}

\section{Some motivations and some perspectives}\label{sec:motiv}

I am not proving new results in this paper: even the quantitative result (i.e.\ Theorem~\ref{thm:exp}) can be deduced from earlier works (see \cite{Hut20}, where Hutchcroft combines OSSS and ghost field techniques).

My initial goal, in the (not for publication) work  \cite{Van22} and the present paper, was to propose a proof of sharpness \textit{without differential inequalities}. As explained in \cite{DMT21} (where the authors propose a new proof -- also without differential inequalities -- of some scaling relations for 2D Bernoulli percolation), finding approaches  that ``[do] not rely on interpretations (using Russo’s formula) of the derivatives of probabilities
of increasing events using so-called pivotal edges" is interesting because ``such formulas are unavailable in most dependent percolation models''. I have to admit that on the one hand I have found a proof of exponential decay without differential inequalities, but I have not succeeded in extending the proof to other percolation models. Gradually, my aim became more to propose a new point of view on sharpness properties, and as simple a proof of a sharpness phenomenon as I could do. During discussions about such motivations, several colleagues asked me whether these techniques could also be used to prove new results for Bernoulli percolation. This was not my aim but the question is legitimate! I think that the answer is yes and I would like to propose four examples.

\paragraph{A. A new inequality for critical exponents.} In \cite{Van22}, we use similar stochastic domination techniques in order to prove an analogue of Theorem \ref{thm:exp} for the diameter. More precisely, we prove that
\begin{equation}\label{eq:van22}
\forall n \ge m, \quad \theta_{2n}\big( p-2\theta_m(p) \big) \le C \frac{\theta_n(p)}{2^{n/m}},
\end{equation}
where $\theta_n(p)$ is the probability that there is a path from $o$ to the sphere of radius $n$ around $o$ (for the graph distance). Let $\eta_1$ and $\nu$ denote the one-arm and the correlation length exponents (so in particular $\eta_1$ is the exponent that describes the quantity $\theta_n(p_c)$, see for instance \cite{DM21} for precise definitions and for other inequalities involving these exponents). \eqref{eq:van22} implies the following inequality:
\begin{equation}\label{eq:in_exp}
\eta_1 \nu \le 1.
\end{equation}
Similar inequalities for critical exponents that describe the volume of clusters can be proven by using Theorem \ref{thm:exp}, and were already proven in \cite{Hut20}. However, \eqref{eq:in_exp} is new to my knowledge, and the reason why previous techniques (such as those from \cite{Hut20}) do not seem to imply \eqref{eq:in_exp} is that the calculations behind them only seem to work when the exponent of the connection probability is at most~$1$ (see e.g.\ \cite[Section~4.1]{Hut20}). The critical exponent that describes $\psi_n(p_c)$ (which is $1/\delta$ in \cite{Hut20}) is always at most $1$, but this is not the case of $\eta_1$, which is $2$ on trees for instance.

\paragraph{B. The existence of a percolation threshold for finite graphs.} In \cite{Eas22}, P.\ Easo proves the existence of a percolation threshold for sequences of finite graphs by using several techniques including stochastic comparison techniques inspired by \cite{Van22}.\footnote{On the other hand, Easo writes: ``[...] we could instead adapt Hutchcroft’s proof of sharpness from \cite{Hut20}, rather than Vanneuville’s new proof. However, the adaption we found uses the universal tightness result from \cite{Hut21} and consequently yields a slightly weaker final bound.''} The present paper is inspired by both \cite{Van22} and \cite{Eas22}.

\paragraph{C. A new inequality for arm events in the plane, and an application to exceptional times in dynamical percolation.} In this paragraph, we consider Bernoulli percolation on a planar (symmetric) lattice. Let $\alpha_k(n)$ denote the probability of the $k$-arm event from $o$ to distance $n$, at the critical parameter (see e.g.\ \cite[Chapter 2]{GS14} for this terminology). The OSSS inequality implies that
\begin{equation}\label{eq:524}
n^2\alpha_4(n)\alpha_2(n) \ge c.
\end{equation}
(See for instance \cite[Chapter 12, Theorem 40]{GS14}.) I believe that -- if one overcomes some technical difficulties related to the arm separation techniques -- one can prove the following by using stochastic domination techniques similar to those proposed in the present paper:
\begin{equation}\label{eq:524delta}
n^2 \alpha_4(n)\alpha_2(n) \ge cn^c.
\end{equation}
The latter inequality is not only a very slight improvement of \eqref{eq:524}. Indeed, as explained in \cite{TV22}, \eqref{eq:524delta} implies that -- if one considers dynamical percolation at the critical parameter -- there exist \textit{exceptional times} when there are both primal and dual unbounded components.

\smallskip

The inequality \eqref{eq:524} is known for very general planar lattices whereas \eqref{eq:524delta} is known (to my knowledge) only for site percolation on the regular triangular lattice (\cite{SW01}) and bond percolation on $\Z^2$ (see \cite{DMT21}, where the authors use parafermionic observables). I believe that one could prove \eqref{eq:524delta} by combining the following two results/observations:
\begin{enumerate*}
\item As proven by Kesten \cite{Kes81}, the size of the near-critical window for planar percolation is $1/(n^2\alpha_4(n))$;
\item The size of the near critical window is much less than $\alpha_2(n)$. How can one prove this? Let us define an exploration of some percolation event by following an interface that separates a primal macroscopic cluster from a dual macroscopic cluster. I believe that results analogous to Lemma \ref{lem:coupl} (together with RSW techniques) imply that the near-critical window is less than the infimum -- on every vertex $x$ and on every path $\gamma$ -- of the probability of the $4$-arm event at $x$, conditionally on the event $\{x \text{ belongs to the interface}$ and $\text{$\gamma$ is the interface stopped at $x$}\}$. This probability is less than the $2$-arm probability in a space that has the topology of the half-plane, which can be strongly expected to be much smaller than $\alpha_2(n)$.
\end{enumerate*}

\paragraph{D. A question based on the above paragraph.} Can one find other contexts (e.g.\ non-planar) where general stochastic domination lemmas such as Lemma~\ref{lem:coupl} are more quantitative than the OSSS inequality?



{\footnotesize
\bibliographystyle{alpha}
\bibliography{ref_new_proof_sharpness}
}

\end{document}